\documentclass{amsart}
\usepackage[utf8]{inputenc}
\usepackage[english]{babel}
\usepackage{amsmath, graphicx}
\usepackage{amsfonts, subcaption}
\usepackage{amssymb}
\usepackage{amsthm}
\usepackage{setspace}
\usepackage[all]{xy}
\usepackage{tikz}

\newtheorem{thm}{Theorem}

\newtheorem{lem}{Lemma}
\newtheorem*{cor*}{Corollary}

\theoremstyle{definition}

\theoremstyle{remark}
\newtheorem{rem}{Remark}
\newtheorem{obv}{Observation}

\makeatletter
\newcommand{\tpitchfork}{%
  \vbox{
    \baselineskip\z@skip
    \lineskip-.52ex
    \lineskiplimit\maxdimen
    \m@th
    \ialign{##\crcr\hidewidth\smash{$-$}\hidewidth\crcr$\pitchfork$\crcr}
  }%
}
\makeatother

\begin{document}
\title{On a potential contact analogue of Kirby move of type 1}

\author{Prerak Deep}
\address{Department of Mathematics, Indian Institute of Science Education and Research Bhopal}
\curraddr{}
\email{prerakd@iiserb.ac.in}
\thanks{}

\author{Dheeraj Kulkarni}
\address{Department of Mathematics,  Indian Institute of Science Education and Research Bhopal}
\curraddr{}
\email{dheeraj@iiserb.ac.in}
\thanks{}



\date{\today}

\dedicatory{}

\begin{abstract}
    In this expository note, we explore the possibility of the existence of Kirby move of type 1 for contact surgery diagrams. 
    In particular, we give the necessary conditions on a contact surgery diagram to become a potential candidate for contact Kirby move of type 1. 
    We observe that there is a collection of contact positive integral surgery diagrams on Legendrian unknots satisfying those conditions.   
\end{abstract}

\maketitle

\section{Introduction}
In the 1960s, Lickorish and Wallace showed that any closed connected oriented 3-manifold admits a surgery presentation in terms of a surgery diagram in $\mathbb{S}^3$. In particular, they proved the following result independently in \cite{L} and \cite{W}.
\begin{thm}[Lickorish--Wallace]\label{LWThm}
        Any closed oriented connected smooth 3-manifold can be obtained by performing a sequence of $(\pm1)$-surgeries on a link in $\mathbb{S}^3$.
\end{thm}
 In 2004, Ding and Geiges proved an analogue result for contact 3-manifolds. More precisely, they proved the following statement in \cite{DG}. 
        \begin{thm}
            Any closed oriented connected contact 3-manifold can be obtained by a sequence of contact $(\pm1)$-surgeries on some Legendrian link in $(\mathbb{S}^3, \xi_{st})$.
        \end{thm}

The above results show the existence of surgery diagrams in the smooth and contact categories, respectively. However, it is easy to notice that for a given smooth 3-manifold, there could be more than one surgery diagram representing it. 
So, it is natural to ask if there is any relation between the two surgery diagrams $L_1$ and $L_2$ corresponding to the same manifold.
Kirby answered this question affirmatively in \cite{RK} for the topological category. Kirby used two types of moves, now known as Kirby moves of type 1 and 2, to obtain $L_1$ from $L_2$, and vice versa.

In \cite{FG}, Ding and Geiges discussed the contact analogue of the Kirby move of type 2. 
However, they do not discuss the analogue of the Kirby move of type 1. 
It is not clear whether one can formulate the analogue of the Kirby move of type 1. 
In 2012, R. Avedek in \cite{RA} showed that for open books, an analogue of Kirby's theorem holds in the contact category. 
In particular, ribbon moves on open books play the role of the Kirby moves, as shown in Theorem 1.10 in \cite{RA}. Hence, one expects an analogue of the Kirby move of type 1 in terms of contact surgery diagrams. 

In this note, we explore a potential analogue of the Kirby move of type 1 in the setting of the contact surgery diagram. 
In particular, we give two necessary conditions in Observation \ref{condonKirbymove} for a contact surgery diagram to be a potential contact analogue of the Kirby move of type 1. 
We classify the candidate contact integral surgery diagrams into two collections denoted by $\mathcal{C}_1$ and $\mathcal{C}_2$. 
In Theorem \ref{MT}, we prove that any surgery diagram in $\mathcal{C}_1$ does not satisfy the necessary conditions. 
In Theorem \ref{C2Thm}, we eliminate a type of surgery diagrams from the collection $\mathcal{C}_2$. 

We would like to mention that Marc Kegel has proved the Theorems \ref{MT} and \ref{C2Thm} in his PhD thesis \cite{MK-thesis}. 
He proved these results while aiming to solve the Legendrian knot complement problem and finding cosmetic contact surgeries on $\mathbb{S}^3$. In contrast, our aim is to find out the possible analogue of the Kirby move of type 1 for contact surgery diagrams.
Although our proofs are essentially the same as given by Kegel, we are writing this article to bring clarity on possible contact Kirby move(s) of type 1. Hence, we do not claim the originality of the results presented in this article.

\section*{Acknowledgements} 
We would like to thank Marc Kegel for his comments on the earlier version of the note and for pointing out the similarities of the note with his PhD Thesis. We would also like to mention that Vera Vértesi is working on finding out a complete set of contact Kirby moves for contact Kirby calculus. The first author is supported by the Prime Minister Research Fellowship-0400216, Ministry of Education, Government of India.  
\section{Preliminaries}

Let $M$ be a smooth 3-manifold. A {\bf \it contact structure} $\xi$ on $M$ is a maximally non-integrable hyperplane field, i.e., for a locally defining 1-form $\alpha$ such that $ker(\alpha)= \xi$ satisfies $\alpha \wedge d\alpha \neq 0$. Such 1-form $\alpha$ is called a contact form. The pair $(M,\xi)$ is called a contact manifold.
On $\mathbb{S}^3 \subset \mathbb{R}^4$, a standard contact structure $\xi_{st}$ can be defined as the kernel of the contact 1-form $\alpha_{st}=\left( x_1dy_1-y_1dx_1+ x_2dy_2-y_2dx_2\right)$.

An embedded circle $K$ in a contact 3-manifold $(M, \xi)$ is called {\it Legendrian knot} if tangential to $\xi$. 
In this note, we work with Legendrian unknot $K$ in $(\mathbb{S}^3, \xi_{st})$. A Legendrian unknot is a topologically trivial Legendrian knot.
Due to Eliashberg--Fraser classification in \cite{EF}, classical invariants Thurston--Bennequin number and rotation number classify Legendrian unknots completely. 
For more details on contact manifolds, Thurston--Bennequin number and rotation number, the reader can refer to \cite{HG-Book}. 

Dehn surgery is an operation to get a new 3-manifold from a given 3-manifold. 
Dehn surgery is performed on a knot $K \subset M$. We cut a thickened neighbourhood $N(K)$ (homeomorphic to $\mathbb{S}^1\times \mathbb{D}^2$) from $M$ and glue a solid torus along the boundary torus via a homeomorphism $h: \partial(\mathbb{S}^1\times \mathbb{D}^2) \to \partial(N(K))$. 
After gluing, we get a 3-manifold, $N= \overline{M\setminus(N(K))} \cup_{h} \mathbb{S}^1\times \mathbb{D}^2$. We call $N$ a surgered manifold obtained from $M$ after performing surgery.

Dehn surgery is completely determined by the image of the meridian $m$, $h(m) = p \mu + q \lambda$ in $\partial(N(K))$, for a given choice of meridian $\mu$ and longitude $\lambda$ in $\partial(N(K))$. 
On $N(K)$, $h(m)$ is a curve that links to $K$ and defines a framing on $K$. 
In particular, Dehn surgery can be determined by {\it framed knot $K$} with framing given by a reduced fraction $\frac{p}{q}$ with $p,q \in \mathbb{Z}$.
For example, we get $\mathbb{S}^3$ after performing Dehn surgery on $\mathbb{S}^3$ along the unknot with framing $\pm1$.
By Theorem \ref{LWThm}, any closed connected oriented smooth 3-manifold $M$ can be presented by a framed link $L$ (link with framed components) in $\mathbb{S}^3$. 
Such a framed link $L$ is called a {\it surgery diagram of $M$}. 
In the surgery diagram, we denote framing on the top of the framed knot $K$.

Let framed links $L_1$ and $L_2$ correspond to the same closed oriented connected smooth 3-manifold $M$; then there is a sequence of Kirby moves of type 1 and 2 to obtain $L_2$ from $L_1$ by [\cite{RK}, Theorem 1].  
We describe the Kirby moves as follows. 
The Kirby move of type 1 is an addition (or deletion) of unknot with $\pm1$ framing to (or from) a surgery diagram (see Figure \ref{fig:FirstKM}). 

\begin{figure}[ht]
    \centering
    \includegraphics[scale=0.7]{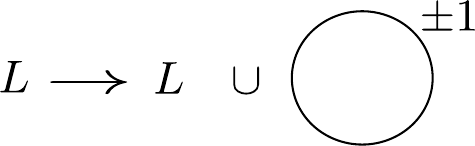}
    \caption{Kirby move of type 1}
    \label{fig:FirstKM}
\end{figure}

The Kirby move of type 2 is sliding of one component of link $L$ over another. In particular, given two components $L_1$ and $L_2$ in ${L}$ with framing $n_1$ and $n_2$ respectively, we may slide $L_1$ over $L_2$ along the longitude $L'_2$ determined by the framing. We replace $L_1 \cup L_2$ by $L'_1 \cup L_2$, where $L'_1= L_1 \#_b L'_2$ (band connected sum of $L_1$ and $L'_2$) (Figure \ref{fig:SecondKM}).  

\begin{figure}[ht]
    \centering
    \includegraphics[scale=0.7]{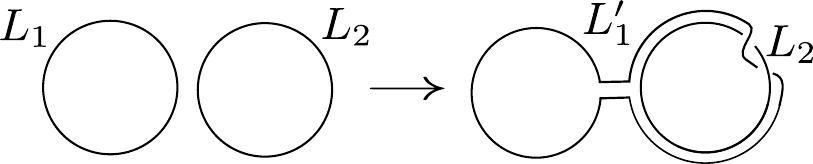}
    \caption{Kirby move of type 2}
    \label{fig:SecondKM}
\end{figure}

The notion of Dehn surgery has been extended to the contact 3-manifold $M$ in \cite{DG}.
We define Dehn surgery on Legendrian knots $K$. 
We call this surgery a {\it contact surgery} on $M$. 
Contact surgery on a contact 3-manifold yields a new contact 3-manifold. 
Moreover, contact surgery can be shown by a surgery diagram consisting of a framed Legendrian link (in its front projection). 
In contact surgery, the framing curve is determined with respect to the meridian and the {\it contact longitude}, i.e. the longitude curve induced by the contact structure on $M$.

\section{Contact Kirby move of type 1}
A reasonable expectation is that a potential contact Kirby move of type 1 would be a Kirby move of type 1 in the topological category. 
We want to explore the contact surgery diagram in $(\mathbb{S}^3, \xi_{st})$, which is \emph{topologically} equivalent to Kirby move of type 1. 

The Kirby move of type 1 consists of a surgery diagram of $(\pm1)$-framed unknot. Thus, any potential contact analogue would consist of a framed Legendrian unknot with a contact framing corresponding to  $(\pm1)$-framing for Dehn surgery. 
Therefore, we explore the relation between the contact framing on Legendrian unknots and their Dehn surgery framing $\pm1$. 
To this end, we consider the standard tubular neighbourhood $N(K)$, which is diffeomorphic to $\mathbb{S}^1\times \mathbb{D}^2$, determined by the Seifert surface of Legendrian unknot $K$ of $tb(K)=-m$. 
Recall from [Section 4.1.3, \cite{KH}], the standard tubular neighbourhood has coordinates $(z,(x,y))$ and contact 1-form $\alpha= cos(2\pi mz)dy-sin(2\pi m z)dx$ with core $K= \mathbb{S}^1 \times {0}=\{(z,(x,y)) \, |\, x=y=0\}$.
On the boundary torus $\partial N(K)$, we use given Seifert framing to identify $\partial N(K)$ with $\mathbb{R}^2/\mathbb{Z}^2$ by taking ${y=0}$ and $ {x=0}$ as the meridian $\mu$ and longitude $\lambda$ curve, respectively. So, the contact longitude $\lambda_c = \lambda-m\mu $.

Let $c$ be the contact framing $n$ curve on boundary torus $\partial N(K)$. Then $c=\lambda_c+ n\mu =\lambda+ (n-m)\mu$. Notice that $c$ is also a Dehn surgery framing $\pm1$ curve on $\partial N(K)$, i.e.  $c= \lambda \pm \mu$. Thus,
\begin{equation}\label{topcond}
    n-m= \pm1 \implies n=m\pm1.  
\end{equation}

We call Equation \ref{topcond} the {\it topological condition} for the existence of a potential contact Kirby move of type 1.

The Kirby move of type 1 does not change the underlying topological manifold. 
It is reasonable to expect its contact analogue to keep the underlying contact 3-manifold unchanged. 
In other words, adding or deleting a framed Legendrian unknot from the contact surgery diagram must not change the underlying contact structure. 
Therefore, the framed Legendrian unknot $K$ with contact framing $``n"$ must yield $(\mathbb{S}^3, \xi_{st})$ back after surgery. 
Moreover, the Bennequin inequality forces $tb(K)<0$, which implies that $0<m\in \mathbb{Z}$. By Equation \ref{topcond}, we get $0\leq n$.
We may write these observations as necessary conditions for a potential contact analogue of Kirby move of type 1.

\begin{obv}\label{condonKirbymove}
The contact surgery diagram, which could be a potential contact Kirby move of type 1, will satisfy the following necessary conditions.
\begin{enumerate}
    \item It consist of a framed Legendrian unknot $K$ with contact framing $n \in \mathbb{N}\cup \{0\}$ satisfying $n=m\pm1$, where $tb(K)=-m$.  
    \item After surgery on $K$, we get $(\mathbb{S}^3, \xi_{st})$ back. 
\end{enumerate}
\end{obv}

We use Equation \ref{topcond} to define two collections $\mathcal{C}_1$ and $\mathcal{C}_2$ of contact surgery diagrams satisfying the above conditions. 
For $j=1,2$; we define $\mathcal{C}_j$ to be the set of contact $n$-surgery diagrams on framed Legendrian unknot $K$ with $tb(K)=-m$ satisfying above two conditions, where $n=m+(-1)^{j}$. 
Now, we state our main theorems. 

\begin{thm}\label{MT}
There is no contact surgery diagram in $\mathcal{C}_1$ satisfying the necessary conditions of potential contact analogue of Kirby move of type 1.
\end{thm}

The collection $\mathcal{C}_2$ consists of the contact $(m+1)$-surgery diagrams on Legendrian unknot with $tb(K)=-m$ satisfying necessary conditions given in Observation \ref{condonKirbymove}. 
By Theorem 5.4 in \cite{MK}, any contact surgery diagram on Legendrian knot $K$ with $tb(K)=-m$ for $(S^3, \xi_{st})$ is coarse equivalent to Legendrian unknot $K$ with $tb(K)=-m$ and $rot(K)=|m-1|$ (like unknot $K$ in the left of the Figure \ref{fig: twosuregrypresentations}). 
It is known that the coarse equivalence of knots is the same as Legendrian isotopy in $(\mathbb{S}^3, \xi_{st})$, and a contact surgery of the same kind on the Legendrian isotopic knots produces the contactomorphic $3$-manifolds. 
Therefore, the collection $\mathcal{C}_2$ reduces to the collection of contact $(m+1)$-surgery diagrams shown in the left of the Figure \ref{fig: twosuregrypresentations} for $m\in \mathbb{N}$.

In addition, we know that any positive integral contact surgery on a Legendrian knot corresponds to exactly two contact structures on the resultant manifold. Each of the contact structures can be expressed by a contact $(\pm1)$-surgery presentation obtained using an algorithm in \cite{DGS} (See Figure \ref{fig: twosuregrypresentations}). 
In the Theorem \ref{C2Thm}, we prove that the contact surgery presentation shown in the bottom right in Figure \ref{fig: twosuregrypresentations} does not produce ($\mathbb{S}^3, \xi_{st}$). 

\begin{figure}
    \centering
    \includegraphics[width=0.7\linewidth]{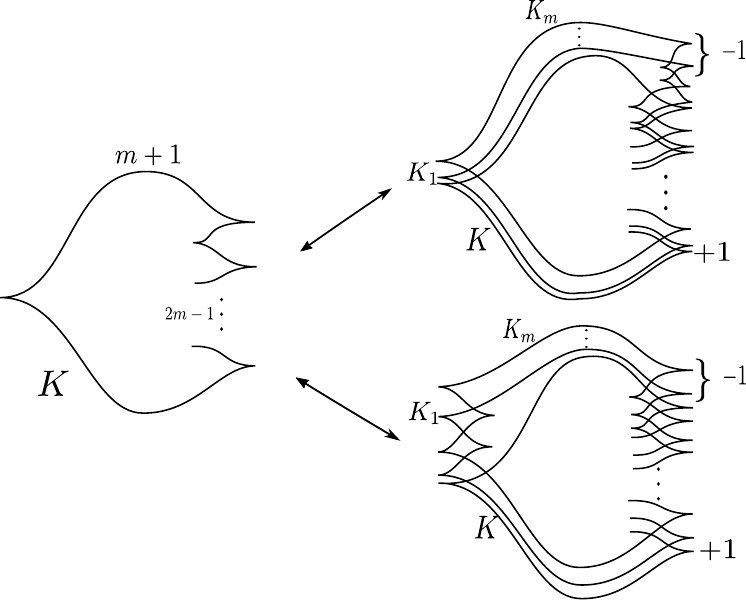}
    \caption{Contact $(\pm1)$-surgery presentations of contact $(m+1)$-surgery on Legendrin unknot $K$ with $tb(K)=-m$}
    \label{fig: twosuregrypresentations}
\end{figure}

\begin{thm}\label{C2Thm}
In $\mathcal{C}_2$, contact $(\pm1)$-surgery presentation given in left to top right in Figure \ref{fig: twosuregrypresentations} does not satisfy the necessary conditions of potential contact analogue.
\end{thm}

\begin{rem}
In Lemma 5.4.4 in \cite{MK-thesis}, contact surgery presentation shown in top right region of Figure \ref{fig: twosuregrypresentations} in $\mathcal{C}_2$ produces the $(\mathbb{S}^3, \xi_{st})$. Further, the lemma proves that any contact surgery diagram of $(\mathbb{S}^3, \xi_{st})$ on a single Legendrian unknot is given by this presentation only.  
\end{rem}

Since one out of the two presentations satisfies the necessary conditions, we denote the corresponding contact surgery diagram on a single Legendrian unknot with {\it tight} written on the top. 
Thus, the collection $\mathcal{C}_2$ further reduces to a collection equal to the set of contact surgery diagrams in Figure \ref{fig:NewDiagram} for any $m\in \mathbb{N}$.  

\begin{figure}
    \centering
    \includegraphics[scale=0.5]{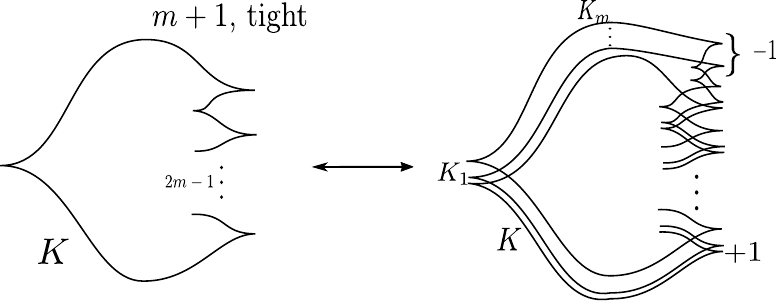}
    \caption{New contact surgery diagram that shows corresponding contact $\pm1$ surgery presentation produce $(\mathbb{S}^3, \xi_{st})$.}
    \label{fig:NewDiagram}
\end{figure}

Therefore, a potential contact analogue of Kirby move of type 1 can be {\it an addition (or deletion) of contact surgery diagram $K$ to (or from) a given integral contact surgery diagram}, where $K \in \mathcal{C}_2$.

\section{Proofs of main theorems}

In this section, we prove our main theorems. For the proofs, we need to discuss some terminologies and a lemma. 

Suppose $K$ is one of the boundary components of some embedded surface $\Sigma$ in $\mathbb{S}^3$. 
Since $\Sigma$ has more than one boundary component, $\Sigma$ is not a Seifert surface of $K$.
Let $N(K)$ denote a standard tubular neighbourhood determined by the Seifert surface of $K$. 
In $N(K)$, $\Sigma$ determines a push-off $K'$ of $K$ on $\partial N(K)$. 
We call $K'$ the framing curve induced by $\Sigma$. 
Now, we can discuss the following Lemma \ref{ChangeInTb} required for the proofs of the main results. 

The more general versions of the following lemma are proved in the literature (the reader may refer to [Lemma 6.6, \cite{P}], [Theorem 5.4, \cite{MK}]). However, we present a different approach than that of \cite{P, MK}. 

\begin{lem}\label{ChangeInTb}
    Let $K$ be a Legendrian unknot in $(\mathbb{S}^3, \xi_{st})$ with $tb(K)=-n$ denoting a contact $(n+1)$-surgery (or, contact $(n-1)$-surgery). If $K_0$ is a topological framing unknot, then Thurston--Bennequin number of $K_0$ changes by $-1$ (or, $+1$) after performing contact surgery on $K$.   
\end{lem}

\begin{proof}
The discussion above shows that a contact $(n\pm1)$-surgery on a given Legendrian unknot $K$ is topologically a $(\pm1)-$surgery.
Thus, topological framing unknot $K_0$ has $lk(K, K_0)=\pm1$, respectively. 

We consider the Hopf link $L=K\cup K_0$. 
Let $S_L$ be a Seifert annulus of $L$. 
The Seifert annulus $S_L$ induces $\pm1$ framing on $K_0$.  
\begin{figure}[ht]
    \centering
    \includegraphics[scale=0.5]{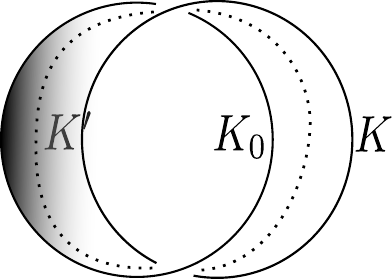}
    \caption{Shaded Seifert annulus $S_L$ with curve $K^{\prime}$ (in dotted line) and boundary curves $K \cup K_0$.}
    \label{fig:K' K_0 K}
\end{figure}

On the boundary of a small tubular neighbourhood $N(K_0)$, we denote $\lambda'$ to be a longitude curve determined by the Seifert disk of $K_0$ and closed curve $\mu'$ to be a meridian such that $<\lambda', \mu'> = H_1(\partial (N(K)))$, after fixing orientation on $K_0$. 
Let $K'= S_L \cap \partial N(K_0)$. 
Since $lk(K_0, K')=\pm1$, we can express $K'=\lambda' \pm \mu'$.

Suppose $K_0$ has $tb(K_0)=-t$ in $(\mathbb{S}^3, \xi_{st})$. 
On $\partial (N(K))$, we may express contact longitude $\lambda'_c  = \lambda'-t \mu'$ in terms of $K'$ as below. 
\begin{equation}\label{tb_new(K_0)}
    \lambda'_c  = \lambda'-t \mu'= (\lambda'\pm \mu') - (t\pm1)\mu' = K'-(t\pm1)\mu'.
\end{equation}

While performing surgery, we remove a thickened neighbourhood $N(K)$ of $K$ such that $\partial (N(K)\cap S_L)= K'$.
After removing $N(K)$, we get a new annulus $S_{L'}:= \overline{S_L \setminus \left(N(K) \cap S_{L}\right)}$ with boundary $\partial S_{L'}= K' \cup K_0$. 
Notice that $lk(K,K')=\pm1$, i.e. $K'$ is a topological framing curve. 
Thus, we attach a meridional disk $D_m$ to $K'$. 
We get an embedded disk $D_0= D_m \cup S_{L'}$ with boundary $K_0$ in the new surgered $\mathbb{S}^3$.

In the new $\mathbb{S}^3$, $D_0$ becomes a Seifert disk of $K_0$.
It determines a longitude $\lambda_{new}$ on the boundary of some thickened neighbourhood $N_{new}(K_0) \subset N(K_0)$ of $K_0$. 
Clearly, $\lambda_{new}$ is isotopic to $K'$. 
Thus, the contact longitude with respect to the $\lambda_{new}$ and $\mu'$ is $K'-(t\pm1)\mu'$ from Equation \ref{tb_new(K_0)}. 
Hence $tb_{new}(K_0)=-(t\pm1)$ in the new $\mathbb{S}^3$. 
\end{proof}

In the proof of Theorem \ref{MT}, we show the existence of an overtwisted disk in the new $\mathbb{S}^3$ obtained after surgery on any surgery diagram in $\mathcal{C}_1$. Recall that an embedded disk $D$ in a contact 3-manifold $(M,\xi)$ is an {\it overtwisted disk},  if it has Legendrian boundary whose surface framing coincides with the contact framing, and in the interior, there is a point with coinciding contact plane and tangent plane to the disk. 

A contact structure $\xi$ on a contact 3-manifold $M$ is called {\it overtwisted} if it contains an overtwisted disk. 
A contact structure $\xi$ is {\it tight} if it is not an overtwisted contact structure. 
For example, a standard contact structure $\xi_{st}$ on $\mathbb{S}^3$ is a tight contact structure on $\mathbb{S}^3$. 
In addition, Eliashberg in \cite{Eli-OverCon} proved that $\mathbb{S}^3$ has a unique tight contact structure up to isotopy, and it is given by $\xi_{st}$.

\begin{proof}[Proof of Theorem \ref{MT}]
Suppose $K \in \mathcal{C}_1$ with contact framing $n=m-1$, where $tb(K)=-m$. If $m=1$, it implies a contact 0-surgery on $K$. Contact 0-surgery yields an overtwisted contact structure on the surgered manifold by definition.  

Now we assume $m>1$. 
Suppose that Legendrian unknot $K$ with $tb(K)=-m$ denotes a contact $(m-1)$-surgery diagram. 
It corresponds to topological $(-1)$-surgery on $K$. Suppose $K_0$ is a topological framing unknot such that it is also a Legendrian unknot with $tb(K_0)=-1$. 

In the proof of Lemma \ref{ChangeInTb}, after surgery, we obtain a new Seifert disk $D_0$ such that $\partial D_0=K_0$ and $tb_{new}(K_0)=-(1-1)=0$. 
In particular, the contact framing of $K_0$ coincides with the surface framing induced by the new Seifert disk $D_0$.

We may perturb $D_0$ in the interior such that the contact plane coincides with the tangent plane at some point. 
Thus, $D_0$ becomes an overtwised disk in the new surgered ${\mathbb{S}}^3$. Hence, the new ${\mathbb{S}}^3$ has an overtwisted contact structure.   
\end{proof}

\begin{figure}[ht]
    \centering
    \includegraphics[scale=0.5]{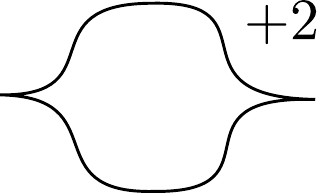}
    \caption{A unique contact integral surgery diagram on unknot with both contact $(\pm1)$-surgery presentation corresponding to  $(\mathbb{S}^3,\xi_{st})$ (refer Example 5.3, \cite{MK}).}
    \label{fig:UniLegSur}
\end{figure}

In the proof of Theorem \ref{C2Thm}, we show that there exists a Legendrian unknot $K$ in the new surgered ${\mathbb{S}}^3$ violating the Bennequin inequality for presentation shown in Figure \ref{fig: twosuregrypresentations} from left to bottom right of every diagram in $\mathcal{C}_2$ except contact $(+2)$-surgery diagram shown in Figure \ref{fig:UniLegSur}. 
For any Legendrian unknot $K$ in $(S^3, \xi_{st})$, the inequality $tb(K)+ |rot(K)| \leq -1$ is known as Bennequin inequality.

\begin{proof}[{Proof of Theorem \ref{C2Thm}}]
Suppose unknot $K$, with $tb(K)=-m \neq -1$, is a contact surgery diagram in $\mathcal{C}_2$. The topological surgery framing on $K$ is $+1$, and the contact framing on $K$ is  $m+1$. 
We draw the topological surgery framing unknot $K_0$ such that it is also a Legendrian unknot $K_0$ of $tb(K_0)=-1$.
Also, note that Bennequin inequality forces $rot(K_0)=0$.
After contact surgery on $K$, we obtain $\mathbb{S}^3$. In the new $\mathbb{S}^3$, the new Thurston--Bennequin number $tb_{new}(K_0)= -(1+1)=-2$ by Lemma \ref{ChangeInTb}. 

We use Lemma 6.6 from \cite{P} to calculate the rotation number of $K_0$ in new $\mathbb{S}^3$. 



Notice that the topological surgery framing unknot $K_0$ has $lk(K, K_0)=+1$ because topological framing push-off takes the same orientation as $K$. 

We start with converting contact $(m+1)$-surgery on unknot $K$ to a sequence of contact $(\pm1)$-surgeries on link $K \cup K_1\cup K_2\cup \ldots\cup K_m$ using an algorithm described in \cite{DGS}. 
In the sequence $K, K_1, K_2, \ldots, K_m$, we perform contact $(+1)$-surgery on $K$ and $(-1)$-surgeries on the rest. 
Since $\frac{m+1}{-m}$ has a continued fraction expansion $[-3, -2, \ldots, -2]$ satisfying the condition given in the algorithm, we add a zigzag to a push-off of $K$ as Figure \ref{fig:K cup K_1} to get $K_1$. 
And $K_i$ is a push-off of $K_{i-1}$, for $2 \leq i \leq m$.
\begin{figure}[ht]
\begin{subfigure}{.4\textwidth}
  \centering
  \includegraphics[scale=0.5]{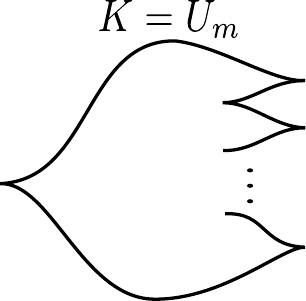}
  \caption{$K=U_m$}
  \label{fig:Um}
\end{subfigure}
\begin{subfigure}{.5\textwidth}
  \centering
  \includegraphics[scale=0.5]{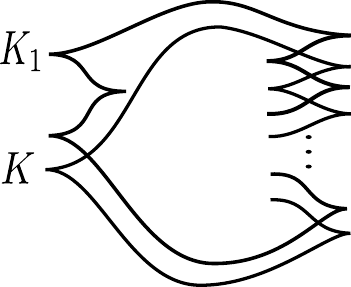}
  \caption{$K\cup K_1$}
  \label{fig:K cup K_1}
\end{subfigure}
\caption{Knot $K$ and its push-off with a zigzag added on the left.}
\end{figure}

We write the matrix $M$ using the formula given in Lemma 6.6 in \cite{P}. 
The entries of $M$ are given by the following 
$$a_{i j}= \begin{cases}
tb(K)+1= -m+1, & \text { for } i=j=1.\\ 
tb(K_{i -1})-1= -m-2, & \text { for } 2 \leq i=j \leq m+1.\\ 
lk(K, K_{j-1})= tb(K)=-m, & \text { for } i=1, 2 \leq j \leq m+1.\\
lk(K, K_{i-1})= tb(K)=-m, & \text { for } j=1, 2 \leq i \leq m+1.\\ 
lk\left(K_{i-1}, K_{j-1}\right)=tb(K_{i-1})=-m-1, & \text { for }  i\neq j \text{ and } 2 \leq i, j \leq m+1.
\end{cases}$$
Thus,
$$M =\left(\begin{array}{ccccc}
-m+1 & -m & -m &  \cdots & -m \\
-m & -m-2 & -m-1 & \cdots & -m-1 \\
\vdots & \vdots & \vdots&  &\vdots \\
-m & -m-1 & -m-1& \cdots  &-m-2
\end{array}\right)_{(m+1) \times(m+1)}.$$
We calculate the inverse 
$$M^{-1} =\left(\begin{array}{ccccc}
m^2 + m +1 & -m &  -m &\cdots & -m \\
-m & 0  & 1 & \cdots & 1 \\
-m & 1 & 0 & \cdots & 1\\
\vdots & \vdots & \vdots & & \vdots \\
-m & 1 &  1 & \cdots  & 0
\end{array}\right)_{(m+1) \times(m+1)}.$$

Since $K_1$ is a push-off of $K$ and $K_i$ is a push-off of $K_{i-1} $, for $2\leq i\leq m$, the linking number $lk(K_0, K_j)= +1$ for $1\leq j \leq m$. 
Therefore, we get linking number column $L= \begin{pmatrix}
    +1 &+1 & \cdots &+1
\end{pmatrix}^T.$ 

Thus, $M^{-1}(L) =
\begin{pmatrix}
    m+1 & -1 & \cdots & -1
\end{pmatrix}^T.$

The rotation numbers of $K_i$ are the same as $K_1$ because they are push-offs of $K_{i-1}$, for $2\leq i\leq m$.
Since $K_1$ is a push off of $K$ with an added zigzag as shown in Figure \ref{fig:K cup K_1}, $rot(K_1)= rot(K)+1= -m+2$. 
Thus, the rotation number column $C =\begin{pmatrix} -m+1 & -m+2 & \cdots & -m+2\end{pmatrix}^T$. Now, we put the above vectors into the formula below. 
\begin{eqnarray*}
    rot_{new}(K_0) &=& rot(K_0) - \left<C, M^{-1}(L)\right>\\
     &=& -\left<C, M^{-1}(L)\right> \quad(\because rot(K_0)=0).\\
    &=& -\{(1-m)(1+m) - m(-m+2)\}\\
    &=&-(1-m^2+m^2-2m) \\
    &=& 2m-1.
\end{eqnarray*}


Clearly, $|2m-1|>2$ for $m\geq2$. Thus, 
$K_0$ does not satisfy the Bennequin inequality because $-2 + |1-2m|>-1 $. Hence, the new $\mathbb{S}^3$ does not have a tight contact structure. 
\end{proof}



\section*{Conclusion}

In \cite{RA}, Avdeck shows that ribbon moves are sufficient to relate contact surgery diagrams corresponding to the same contact 3-manifold. 
Therefore, asking whether contact surgery diagram in $\mathcal{C}_2$ can be realised as a ribbon move is natural. 
We observe that the potential contact Kirby move of type 1 cannot be converted into a ribbon move as follows. 
Ribbon moves are defined in \cite{RA} as relations between mapping classes. 
We prove it for the contact $(+2)$-surgery diagram on unknot and the general case will follow similarly. 
First, we convert contact $(+2)$-surgery into a $(\pm1)$-surgery link $L$ using an algorithm given in \cite{DGS}. We construct a Legendrian graph $\overline{L}$ of surgery link $L$ and then the ribbon surface $R_{\overline{L}}$ of $\overline{L}$ following Algorithm 2 from \cite{RA}. 
\begin{figure}[ht]
    \centering
    \includegraphics[scale=0.7]{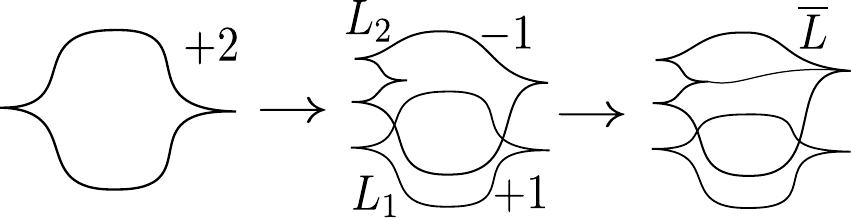}
    \caption{contact $(+2)$-surgery in a sequence of $(\pm1)-$surgery link $L$ and Legendrian graph $\overline{L}$ of $L$.}
    \label{fig:LegGraph}
\end{figure}

In $R_{\overline{L}}$, the surgery link $L$ denotes a mapping class $D_{L_1}^{-1} \circ D_{L_2}^{+1} \in MCG(R_{\overline{L}}, \partial R_{\overline{L}})$. 
But $\psi^{-1}\circ D_{L_1}^{-1} \circ D_{L_2}^{+1} \circ \psi \neq Id_{R_{\overline{L}}}$ for any mapping class $\psi \in MCG(R_{\overline{L}}, \partial R_{\overline{L}})$. 
Hence, the addition or deletion of contact $(+2)$-surgery unknot does not describe a ribbon move.
We can generalise this argument for any contact surgery diagram in $\mathcal{C}_2$. Thus, an addition or a deletion of contact surgery diagram $K \in \mathcal{C}_2$ to or from a given integral contact surgery diagram does not describe a ribbon move. 
Therefore, it remains unclear whether the potential candidates in $\mathcal{C}_2$ play the role of a contact Kirby move of type 1.

\bibliographystyle{plain}
\bibliography{Reference}

\end{document}